\documentclass[12pt,reqno]{article}

\usepackage[usenames]{color}
\usepackage{amssymb}
\usepackage{graphicx}
\usepackage{amscd}

\usepackage[colorlinks=true,
linkcolor=webgreen,
filecolor=webbrown,
citecolor=webgreen]{hyperref}

\definecolor{webgreen}{rgb}{0,.5,0}
\definecolor{webbrown}{rgb}{.6,0,0}

\usepackage{color}
\usepackage{fullpage}
\usepackage{float}

\usepackage{graphics,amsmath,amssymb}
\usepackage{amsthm}
\usepackage{amsfonts}
\usepackage{latexsym}
\usepackage{epsf}

\newcommand{\seqnum}[1]{\href{http://oeis.org/#1}{\underline{#1}}}
\begin{document}
	
	\newenvironment{dedication}
	{\vspace{6ex}\begin{quotation}\begin{center}\begin{em}}
				{\par\end{em}\end{center}\end{quotation}}
			
	\theoremstyle{plain}
	\newtheorem{theorem}{Theorem}
	\newtheorem{corollary}[theorem]{Corollary}
	\newtheorem{lemma}[theorem]{Lemma}
	\newtheorem{proposition}[theorem]{Proposition}
	
	\theoremstyle{definition}
	\newtheorem{definition}[theorem]{Definition}
	\newtheorem{example}[theorem]{Example}
	\newtheorem{conjecture}[theorem]{Conjecture}
	
	\theoremstyle{remark}
	\newtheorem{remark}[theorem]{Remark}
	
	\begin{center}
		\vskip 1cm{\LARGE\bf Transcendental Infinite Products Associated with the $\pm 1$ Thue-Morse Sequence}
		\vskip 1cm
		\large
		L\'aszl\'o T\'oth\\
		Rue des Tanneurs 7 \\
		L-6790 Grevenmacher \\
		Grand Duchy of Luxembourg \\
		\href{mailto:uk.laszlo.toth@gmail.com}{\tt uk.laszlo.toth@gmail.com}
	\end{center}
	\begin{dedication}
		To my daughter Ali\'enor, on the occasion of her first birthday
	\end{dedication}
	\vskip .2 in
	
	\begin{abstract}
		Infinite products associated with the $\pm 1$ Thue-Morse sequence whose value is rational or algebraic
		irrational have been studied by several authors. In this short note we prove three new infinite product identities involving $\pi$, $\sqrt2$ and the $\pm 1$ Thue-Morse sequence, building upon a result by Allouche, Riasat and Shallit. We then use our method to find a new expression for a product appearing within the context of the Flajolet-Martin constant.
	\end{abstract}

	\section{Introduction}
	
	Let $\epsilon_n$ denote the Thue-Morse sequence on the alphabet $\{1,-1\}$. Thus, $\epsilon_n$ begins $1,-1,-1,1,$ $-1,1,1,-1,\ldots$. Infinite products involving $\epsilon_n$ have been extensively studied in literature; perhaps the most well-known is the Woods-Robbins product,
	\begin{displaymath}
	\prod_{n\geq0} \left( \frac{2n+1}{2n+2} \right) ^{\epsilon_n} = \frac{\sqrt2}{2}.
	\end{displaymath}
	There are many ways to prove this identity, for instance by splitting the product $\prod_{n\geq1} \left( \frac{n}{n+1} \right) ^{\epsilon_n}$ into odd and even indexes then using the relations $\epsilon_{2n}=\epsilon_n$ and $\epsilon_{2n+1}=-\epsilon_n$ (Allouche, Mend\`es France and Peyri\`ere \cite{Allouche00}), or by analytically continuing then differentiating the Dirichlet series $\sum_{n\geq0} \frac{\epsilon_n}{(n+1)^s}$ at $s=0$ (Allouche and Cohen \cite{AlloucheCohen85}). Numerous generalizations of the identity above are available. In a $1987$ paper, Allouche, Cohen, Mend\`es France and Shallit \cite{AlloucheShallit87} prove identities of the type 
	\begin{align*}
	&\prod_{n\geq0} \left( \frac{2n+1}{2n+2} \right) ^{-(3/2)^{t_n}} = 2^{-2/5}, \\
	&\prod_{n\geq0} \left( \frac{2n+1}{2n+2} \right) ^{t_n\epsilon_n} = 2^{1/4},
	\end{align*}
	where $t_n$ denotes the usual Thue-Morse sequence. More recently, Allouche and Sondow \cite{AlloucheSondow08} proved that
	\begin{displaymath}
	\prod_{n\geq0} \left( \frac{Bn+1}{Bn+2} \right) ^{(-1)^{t_B(n)}} = \frac{1}{\sqrt B},
	\end{displaymath}
	where $t_B(n)$ denotes the number of ones in the $B$-ary expansion of the integer $n$. Other ``curious'' infinite product identities have been found. Allouche, Riasat and Shallit \cite{Riasat19} prove an abundance of identities associated with $\epsilon_n$, such as 
	\begin{align*}
	&\prod_{n\geq0} \left( \frac{4n+1}{4n+3} \right) ^{\epsilon_n} = \frac{1}{2}, \\
	&\prod_{n\geq0} \left( \frac{(n+1)(4n+5)}{(n+2)(4n+3)} \right) ^{\epsilon_n} = 1,
	\end{align*}
	as well as others involving $t_n$ and the Gamma function, such as 
	\begin{displaymath}
	\prod_{n\geq0} \left( \frac{(4n+1)(4n+4)}{(4n+2)(4n+3)} \right) ^{t_n} = \frac{\pi^{3/4} \sqrt2}{\Gamma(1/4)}
	\end{displaymath}
	and many others. 
	
	No transcendental valued infinite product of this type with exponents equal to the $\pm1$ Thue-Morse sequence is currently known. In this paper, we prove two new infinite product identities involving this sequence:
	\begin{theorem} \label{theo-prod}
		We have
	\begin{align*}
	&\prod_{m\geq1} \prod_{n\geq1} \left( \frac{(4 m^2 + n - 2) (4 m^2 + 2 n - 1)^2}{4 (2 m^2 + n - 1) (4 m^2 + n - 1) (2 m^2 + n)} \right) ^{\epsilon_n} = \frac{\pi}{2}, \\
	&\prod_{m\geq0} \prod_{n\geq1} \left( \frac{(16 m (m + 1) + n + 2) (16 m (m + 1) + 2 n + 3)^2}{4 (8 m (m + 1) + n + 1) (8 m (m + 1) + n + 2) (16 m (m + 1) + n + 3)} \right) ^{\epsilon_n} = \sqrt 2.
	\end{align*}
	\end{theorem}
	We then obtain a third identity as a combination of the two above. Finally, we end this note with a few words on the Flajolet-Martin constant and use our methods to find another expression for the product
	$$
	\prod_{n\geq1} \left( \frac{(4n+1) (4n+2)}{4n (4n+3)} \right) ^{\epsilon_n}.
	$$	
	
	\section{Our results}
	
	We begin by recalling a result by Allouche, Riasat and Shallit \cite{Riasat19}, a key ingredient of our proof.
	\begin{lemma} \label{lem-shallit}
		Let $a,b\in\mathbb{C} \backslash \{-1,-2,-3,\ldots\}$. Then
		\begin{displaymath}
		\prod_{n\geq1} \left( \frac{(n+a)(2n+a+1)(2n+b)}{(2n+a)(n+b)(2n+b+1)} \right) ^{\epsilon_n} = \frac{b+1}{a+1}.
		\end{displaymath}
	\end{lemma}
	Note that this Lemma appears as Corollary 2.3 $(i)$ of the aforementioned work. We are now able to prove our identities presented earlier.
	\begin{proof} [Proof of Theorem \ref{theo-prod}]
		Let $m\geq1$ denote a positive integer. Now taking $a=4m^2-2$ and $b=4m^2-1$ with Lemma \ref{lem-shallit}, we have
		\begin{displaymath}
		\prod_{n\geq1} \left( \frac{(n+4m^2-2)(2n+4m^2-1)(2n+4m^2-1)}{(2n+4m^2-2)(n+4m^2-1)(2n+4m^2)} \right) ^{\epsilon_n} = \frac{4m^2}{4m^2-1}.
		\end{displaymath}
		Taking the product over all $m\geq1$ on both sides, we notice that the right hand side is Wallis' classic identity for $\pi/2$. Simplifying the left hand side yields the first identity as claimed. Now applying $a=(4m+1)(4m+3)-1$ and $b=(4m+2)^2-1$ to Lemma \ref{lem-shallit}, but this time for positive integer $m\geq0$, the product on the left hand side becomes
		\begin{displaymath}
		\prod_{n\geq1} \left( \frac{(n+(4m+1)(4m+3)-1) (2n+(4m+1)(4m+3)) (2n+(4m+2)^2-1)}{(2n+(4m+1)(4m+3)-1) (n+(4m+2)^2-1) (2n+(4m+2)^2)} \right) ^{\epsilon_n},
		\end{displaymath}
		while the right hand side, $\frac{(4m+2)^2}{(4m+1)(4m+3)}$, reminds us of Catalan's identity for $\sqrt 2$. Thus taking the product over all $m\geq0$ on both sides and simplifying the product expression, we prove the second claim.
	\end{proof}
	At this point one may notice a relationship between the first product and the second by looking at the terms for odd and even $m$. This leads us to a third identity.
	\begin{corollary}
		\begin{displaymath}
		\prod_{m\geq1} \prod_{n\geq1} \left( \frac{(16 m^2 + n - 2) (16 m^2 + 2 n - 1)^2}{4 (8 m^2 + n - 1) (16 m^2 + n - 1) (8 m^2 + n)} \right) ^{\epsilon_n} = \frac{\pi \sqrt2}{4}.
		\end{displaymath}
	\end{corollary}
	\begin{proof}
		We begin by splitting the product 
		\begin{displaymath}
		\prod_{m\geq1} \prod_{n\geq1} \left( \frac{(4 m^2 + n - 2) (4 m^2 + 2 n - 1)^2}{4 (2 m^2 + n - 1) (4 m^2 + n - 1) (2 m^2 + n)} \right) ^{\epsilon_n} = \frac{\pi}{2}
		\end{displaymath}
		into odd and even $m$. We notice that for odd indexes (i.e., $2m+1$), we have
		\begin{displaymath}
		\left( \frac{(16 m (m + 1) + n + 2) (16 m (m + 1) + 2 n + 3)^2}{4 (8 m (m + 1) + n + 1) (8 m (m + 1) + n + 2) (16 m (m + 1) + n + 3)} \right) ^{\epsilon_n},
		\end{displaymath}
		which is exactly the inner product in the second identity of Theorem \ref{theo-prod}. Thus, dividing  the first equation in Theorem \ref{theo-prod} by the second yields our third identity.
	\end{proof}

	\section{On the Flajolet-Martin constant}
	
	An interesting infinite product involving the $\pm 1$ Thue-Morse sequence appeared in a $1985$ paper by Flajolet and Martin \cite{FlajoletMartin85}:
	\begin{displaymath}
	R:=\prod_{n\geq1} \left( \frac{(4n+1) (4n+2)}{4n (4n+3)} \right) ^{\epsilon_n},
	\end{displaymath}
	intimately connected with the so-called ``Flajolet-Martin constant'' $\varphi := 2^{-1/2} e^\gamma \frac23 R$. The arithmetical nature of $\varphi$, and thus of $R$, remains a mystery. While it is easy to see that $R$ can be obtained from the product $\prod_{n\geq1} \left( \frac{2n}{2n+1} \right) ^{\epsilon_n}$ by splitting it into odd and even indexes and using the fact that $\epsilon_{2n}=\epsilon_{n}$ and $\epsilon_{2n+1}=-\epsilon_{n}$, not much else is known about it. In the remaining paragraphs of this note, we find another expression for $R$ using the same technique we used above.
	\begin{theorem}
		Let $R$ denote the product $\prod_{n\geq1} \left( \frac{(4n+1) (4n+2)}{4n (4n+3)} \right) ^{\epsilon_n}$. Then $R=$
		\begin{displaymath}
		\prod_{m\geq1} \prod_{n\geq1} \left( \frac{(8 m^2 + 6 m + n) (4 m (4 m + 3) + n - 1) (4 m (4 m + 3) + 2 n + 1)}{(8 m^2 + 6 m + n + 1) (4 m (4 m + 3) + n + 1) (4 m (4 m + 3) + 2 n - 1)} \right) ^{\epsilon_n \epsilon_m}.
		\end{displaymath}
	\end{theorem}
	\begin{proof}
		Let $m\geq1$ be a positive integer. We take Lemma \ref{lem-shallit} with $a=4m (4m+3) -1$ and $b=(4m+1) (4m+2)-1$, so that our product equals $\frac{(4m+1) (4m+2)}{4m (4m+3)}$. Thus the left hand side becomes:
		\begin{displaymath}
		\prod_{n\geq1} \left( \frac{(16 m^2 + 12 m + n - 1) (8 m^2 + 6 m + n) (16 m^2 + 12 m + 2 n + 1)}{(16 m^2 + 12 m + 2 n - 1) (16 m^2 + 12 m + n + 1) (8 m^2 + 6 m + n + 1)} \right) ^{\epsilon_n}.
		\end{displaymath}
		Now taking the product over all $m\geq1$ and flipping each term following the sign of the $\epsilon_m$ sequence yields
		\begin{displaymath}
		\prod_{m\geq1} \left( \prod_{n\geq1} \left( \frac{(16 m^2 + 12 m + n - 1) (8 m^2 + 6 m + n) (16 m^2 + 12 m + 2 n + 1)}{(16 m^2 + 12 m + 2 n - 1) (16 m^2 + 12 m + n + 1) (8 m^2 + 6 m + n + 1)} \right) ^{\epsilon_n} \right) ^{\epsilon_m}.
		\end{displaymath}
		Since the exponents are always one of $\{1,-1\}$, we can join them together as $\epsilon_n \epsilon_m$, and after simplifying the expression within the double product our proof is complete.
	\end{proof}

	\section{Final remarks}
	The method we used to derive our results in the previous sections is easily adapted to other constants admitting an infinite product representation with rational terms. We state this remark as follows.
	\begin{remark}
		Considering the infinite product $\prod_{m\geq m_0}^{} b_m/a_m = C$, where $a_m$ and $b_m$ are nonzero integers for all $m\geq m_0$ and $C$ is a closed form expression comprised of known constants, we have
		\begin{displaymath}
		\prod_{m\geq m_0} \prod_{n\geq1} \left( \frac{(n+a_m-1) (2n + a_m) (2n+b_m-1)}{(2n + a_m-1) (n+b_m-1) (2n+b_m)} \right)^{\epsilon_{n}} = C.
		\end{displaymath}
	\end{remark}
	Note that one can prove this easily by setting $a=a_m-1$ and $b=b_m-1$ in Lemma \ref{lem-shallit} and then taking the product over all $m\geq m_0$. A few examples that quickly come to mind can be found in the book of Borwein, Bailey and Girgensohn \cite[pp. 4--6]{Borwein04}:
	\begin{align*}
	\prod_{n=2}^{} \frac{n^2-1}{n^2+1} &= \pi \cosh (\pi), \\
	\prod_{n=2}^{} \frac{n^3-1}{n^3+1} &= \frac23, \\
	\prod_{n=2}^{} \frac{n^4-1}{n^4+1} &= \frac{\pi \sinh (\pi)}{\cosh (\sqrt 2 \pi) - \cos (\sqrt 2 \pi)}, \\
	\end{align*}
	which result in some amusing identities involving the $\epsilon_n$ sequence.

\bigskip
\hrule
\bigskip

\noindent 2010 {\it Mathematics Subject Classification}: Primary 11A63; Secondary 68R15, 11Y60, 11B85. \\
\noindent \emph{Keywords: } Thue-Morse sequence, infinite product, Flajolet-Martin constant.

\bigskip
\hrule
\bigskip

\noindent (Concerned with sequences
\seqnum{A106400} and
\seqnum{A010060}.)

\bigskip
\hrule
\bigskip

\end{document}